\documentclass[12pt]{amsart}

\usepackage{amsmath,amssymb,txfonts}
\usepackage{amssymb}
\usepackage{amsxtra}
\usepackage{amsthm, color}
\usepackage{txfonts}
\usepackage{graphicx}
\usepackage{color}
\usepackage{enumitem}
\usepackage{times}

\usepackage[cp1252]{inputenc}

\usepackage{graphicx}

\usepackage[active]{srcltx}

\newcommand{\rmv}[1]{}

\numberwithin{equation}{section}

\newcommand{\HH}{\mathcal{H}}
\newcommand{\HK}{\mathcal{K}}
\newcommand{\HM}{\mathcal{M}}

\newcommand{\HD}{\mathcal{D}}
\newcommand{\HS}{\mathcal{S}}
\newcommand{\HB}{\mathcal{B}}
\newcommand{\HN}{\mathcal{N}}

\newcommand{\D}{\mathbb{D}}

\newcommand{\C}{\mathbb{C}}

\newcommand{\R}{\mathbb{R}}
\newcommand{\Z}{\mathbb{Z}}

\newcommand{\T}{\mathbb{T}}

\newcommand{\ran}{\mathrm{ran \ }}

\newcommand{\rank}{\mathrm{rank \ }}

\newcommand{\La}{\langle}

\newcommand{\Ra}{\rangle}

\theoremstyle{plain}
\newtheorem{theorem}{Theorem}[section]
\newtheorem{lemma}[theorem]{Lemma}
\newtheorem{remark}[theorem]{Remark}

\newtheorem{proposition}[theorem]{Proposition}
\newtheorem{corollary}[theorem]{Corollary}

\theoremstyle{definition}
\newtheorem{example}[theorem]{Example}

\begin{document}

\title[Finite rank de~Branges--Rovnyak spaces]{On Dirichlet-type and $n$-isometric shifts in finite rank de~Branges--Rovnyak spaces}
\author{Shuaibing Luo and Eskil Rydhe}

\address{School of Mathematics, Hunan University, Changsha, Hunan, 410082, PR China}
\email{sluo@hnu.edu.cn}

\address{Centre for Mathematical Sciences, Lund University, Box 118, SE-221 00 Lund, Sweden}
\email{eskil.rydhe@math.lu.se}

\address{}
\email{}

\thanks{S. Luo was supported by NNSFC (12271149), Natural Science Foundation of Hunan Province (2024JJ2008). E.~Rydhe was supported by VR grant 2022-04307.}
\date{}
\begin{abstract}
This paper studies the function spaces $\HD(\mu)$ by Richter and Aleman, and $\HD_{\vec{\mu}}$ by the second author.

It is known that the forward shift $M_z$ is bounded and expansive on $\HD(\mu)$, and therefore $\HD(\mu)$ coincides with a de~Branges--Rovnyak space $\HH[B]$. We show that such a $B$ is rational if and only if $\mu$ is finitely atomic, and this happens exactly when the corresponding defect operator has finite rank. We also outline a method for calculating the reproducing kernel of $\HD(\mu)$ for finitely atomic $\mu$.

Similarly, we characterize the allowable tuples $\vec{\mu} = (\frac{|dz|}{2\pi}, \mu_1, \ldots, \mu_{n-1})$ such that $M_z$ on $\HD_{\vec{\mu}}$ is expansive with finite rank defect operator. This investigation provides many interesting examples of normalized allowable tuples $\vec{\mu}$.
\end{abstract}

\maketitle


\section{Introduction}

Throughout this paper, $T$ will denote a bounded linear operator on a complex separable Hilbert space $\HH$. In the situations we consider, $T$ also satisfies the following assumptions.
\begin{itemize}
	\item[(A1)] $T$ is \emph{analytic}, i.e. $\bigcap_{n\geq 1}T^n\HH=\{0\}$.
	\item[(A2)] $T$ is \emph{expansive}, i.e. $\|Tx\|\ge \|x\|$ for all $x\in\HH$.
	\item[(A3)] $\dim\ker T^* = 1$.
	\item[(A4)] $T$ is \emph{cyclic} with respect to any non-zero $e\in\ker T^*$.
\end{itemize}
It is known that $T$ satisfies (A1)--(A3) if and only if $T$ is unitarily equivalent to $(M_z,\HH[B])$, the \emph{forward shift} acting on a reproducing kernel Hilbert space (RKHS) of de~Branges--Rovnyak type~\cite[Theorem~4.6]{LGR}. A brief introduction to various concepts and notation will be given in the next section.

We are motivated by the study of certain classes of operators, namely \emph{Dirichlet type operators} and \emph{higher order isometries}. For this purpose, we define the higher order \emph{defect operators} $\{\Delta^{(n)}\}_{n=0}^\infty$ by $\Delta^{(0)}:=I$, and $\Delta^{(n)}:=T^*\Delta^{(n-1)}T-\Delta^{(n-1)}$ for $n\ge 1$. Note that $T$ is expansive if and only if the standard defect operator $\Delta:=\Delta^{(1)}\ge 0$.

The operator $T$ is of \emph{Dirichlet type} if it satisfies (A1) and (A3), together with the alternating defect condition $(-1)^n\Delta^{(n)}\le 0$ for all $n\ge 1$. The alternating defect condition clearly implies (A2). A result by Aleman~\cite{Al93} states that $T$ is of Dirichlet type if and only if $T$ is unitarily equivalent to $(M_z,\HD(\mu))$, where $\HD(\mu)$ is the superharmonically weighted Dirichlet space associated to some finite positive measure $\mu$ on the closed unit disc $\overline{\D}$. Since polynomials are dense in $\HD(\mu)$, Dirichlet type operators satisfy (A4).

It is easy to see that $\HD(\mu)$ is a RKHS, but at the same time notoriously difficult to identify the reproducing kernel associated to a given $\mu$. Another result by Aleman and Malman~\cite{AM19} implies that since the \emph{backward shift} $L$ is a contraction on $\HD(\mu)$, the kernel is in fact of de~Branges--Rovnyak type, i.e. $\HD(\mu)$ is \emph{equal to} $\HH[B]$ for some function $B$ in the (normalized) \emph{Schur class} $\HS_0(\ell^2,\C)$. Our first result is a qualitative statement about the relation between $\HD(\mu)$ and $\HH[B]$ in the case where $(M_z,\HD(\mu))$ has finite rank defect operator.
\begin{theorem}\label{theorem:HBDmuequal}
	Let $\mu$ be a finite positive Borel measure on $\overline{\D}$, $T=(M_z,\HD(\mu))$, and $\Delta=T^*T-I$. Take $B\in\HS_0(\ell^2,\C)$ such that $\HH[B]=\HD(\mu)$. The following are equivalent:
	\begin{enumerate}
		\item $\mu$ is finitely atomic.
		\item $\rank\Delta < \infty$.
		\item $B$ is rational.
	\end{enumerate}
	If either condition holds, then $\mu=\sum_{j=1}^nc_j\delta_{\lambda_j}$, where $n=\rank\Delta= \rank B = \deg B$, $c_1,\ldots,c_n>0$, and $\lambda_1,\ldots,\lambda_n\in \overline{\D}$ are distinct.
\end{theorem}
As a complementary result, a given space $\HH[B]$ coincides with $\HD(\mu)$ for some discrete measure if and only if the defect operator of $(M_z,\HH[B])$ takes a particular form.
\begin{theorem}\label{dirichlettype}
Suppose $B\in\HS_0(\ell^2,\C)$ is such that $M_z\HH[B]\subseteq\HH[B]$ and polynomials are dense in $\HH[B]$. Let $T=(M_z,\HH[B])$ and $\Delta = T^*T - I$. The following are equivalent:
\begin{enumerate}
	\item There exists a finite positive discrete measure $\mu$ on $\overline{\D}$ such that $\HD(\mu)=\HH[B]$.
	\item There exists $c_1, c_2, \ldots>0$ and distinct $\lambda_1, \lambda_2,\ldots\in\overline{\D}$ such that $\Delta = \sum_{i=1}^\infty c_i K_{\lambda_i}^B\otimes K_{\lambda_i}^B$, where the sum converges in the strong operator topology.
\end{enumerate}
\end{theorem}
We write $b$ in place of $B$ in the case where $B$ is scalar-valued. The problem of identifying certain $\HH(b)$ with local Dirichlet spaces was indicated in~\cite{RS91}, pioneered by Sarason~\cite{Sa97}, and then further studied in \cite{CGR10, CR13, EFKKMR16}. As a demonstration of our results, we recover the following result, which is a reformulation of \cite[Theorem~5.1]{EFKKMR16}.
\begin{corollary}\label{pointmasscaseCopy}
	Let $b\in\HS_0(\C,\C)\setminus \{0\}$. The following are equivalent:
	\begin{enumerate}
		\item There exists a finite positive measure $\mu$ on $\overline{\D}$ such that $\HD(\mu) = \HH(b)$.
		\item There exists $\beta,\gamma\in\D$ such that $0<|\gamma|\le 1-|\beta|$ and $b(z)=\frac{\gamma z}{1-\beta z}$.
	\end{enumerate}
	If either condition holds, then $\mu = c\delta_\lambda$ for some $c>0$ and $\lambda\in\overline{\D}$.
\end{corollary}
\begin{proof}
	Assume $(i)$. By Theorem \ref{theorem:HBDmuequal}, we have that $b$ is rational of degree $1$ with $b(0)=0$. So $b(z)=\frac{\gamma z}{1-\beta z}$ for some $\beta,\gamma\in\C$. Since $b$ is analytic on $\D$ and $\sup_{z\in\D}|b(z)|\le 1$, we obtain $\beta\in\D$ and $0<|\gamma|\le 1-|\beta|$. Moreover, $M_z\HH(b)\subseteq\HH(b)$, which is equivalent to the ``non-extremal'' condition that $\log (1-|b|^2)$ is integrable on the unit circle. This implies $\gamma\in\D$.
	
	Assume $(ii)$. Then $M_z\HH(b)\subseteq\HH(b)$. By \cite[Theorem 18.22]{FM162}, $L^{\ast}=T-b\otimes Lb$. So $T^{\ast}=L+Lb\otimes b$ and $\Delta = T^*T - I = (1+\left\Vert b\right\Vert ^{2})Lb\otimes Lb$. Since  $Lb = \frac{\gamma}{1 - \beta z}$, it follows that $T^* (Lb) = \alpha (Lb), \alpha = \beta + \langle Lb, b\rangle$. Thus the smallest closed $T^*$-invariant subspace $[\ran \Delta]_{T^*}$ containing $\ran \Delta$ is  $\C (Lb)$. So $Lb$ is a reproducing kernel in $\HH(b)$, see the remark after Lemma \ref{lemma:MateFunctionAndPhi}.
	Then Theorem~\ref{dirichlettype} implies the existence of the discrete measure $\mu$.
	
	The final assertion is immediate from Theorem~\ref{theorem:HBDmuequal}.
\end{proof}
\begin{remark}
	In the corollary, the case $\lambda\in\T$ originally studied by Sarason~\cite{Sa97} corresponds exactly to $|\gamma|=1-|\beta|$, see also \cite{KZ15}.
\end{remark}

We now direct our attention to the so-called higher order isometries studied by Agler and Stankus~\cite{AS95}. The operator $T$ is called an \emph{$n$-isometry} whenever $\Delta^{(n)}=0$. The case $n=2$ was originally studied by Agler~\cite{Agler90}. It is well-known that if $T$ is a $2$-isometry satisfying (A1) and (A3), then (A2) and (A4) are also satisfied; in fact $T$ is of Dirichlet type. For $n\ge 3$, these implications no longer hold. Hence the explicit mention of (A2) and (A4).

An operator $T$ is an $n$-isometry satisfying (A1), (A2), (A4) if and only if $T$ is unitarily equivalent to $(M_z,\HD_{\vec{\mu}})$~\cite{Ry19}. Here, $\vec{\mu}=(\mu_0,\ldots,\mu_{n-1})$ is an allowable $n$-tuple of distributions on the unit circle $\T$, with $\mu_0$ being normalized Lebesgue measure, and $\HD_{\vec{\mu}}$ is the corresponding harmonically weighted space of analytic functions. If we add assumption (A3), then $\HD_{\vec{\mu}}$ coincides with some $\HH[B]$, similar to the case of Dirichlet type operators.

It is rather plain that if $\vec{\mu}$ consists exclusively of non-negative measures, then $(M_z,\HD_{\vec{\mu}})$ is expanding, see~\cite[Remark~6.9]{Ry19}. We shall see that if $n\ge 3$, then this excludes the possibility that $\rank\Delta <\infty$, see Corollary~\ref{infiniterank}. In fact, we obtain a complete description of allowable tuples for which $T=(M_z,\HD_{\vec{\mu}})$ satisfies (A1)--(A4) and has finite rank defect operator. In the interest of space, we state only the rank one case here, and refer to Theorem~\ref{theorem:H[B]=Dvecmu} for the general case.

\begin{theorem}\label{theorem:H(b)=Dvecmu}
	Suppose $b\in\HS_0(\C,\C)$ satisfies \eqref{eq:MateExistenceCondition}, and that $T=(M_z,\HH(b))$ is a strict $n$-isometry for some $n\in\Z_{\ge 1}$. Then $n=2m$ for some $m\in\Z_{\ge 1}$. Moreover, there exists $\lambda\in\T$ and a polynomial $p(z)=\sum_{j=0}^{m-1}c_jz^j$ with the following properties: $p(\lambda)\ne 0$, and if $\vec{\mu}=(\frac{|dz|}{2\pi},\mu_1,\ldots,\mu_{2m-1})$ is given by
	\begin{multline}\label{eq:H(b)=Dvecmu}
		\hat \mu_{i}(k)
		=
		\overline{\hat \mu_{i}(-k)}
		\\
		=
		\overline{\lambda}^k\sum_{j_1,j_2=0}^{m-1}c_{j_1}\overline{c_{j_2}}\lambda^{j_1-j_2}
		\sum_{l=0}^{i-1}(-1)^{i-1-l}\binom{i-1}{l}\binom{l+j_1}{m-1}\binom{l+k+j_2}{m-1},
	\end{multline}
	where $k\in\Z_{\geq 0}$ and $1\le i\le 2m-1$, then $\HH(b)=\HD_{\vec{\mu}}$.
	
	Conversely, if $m\in\Z_{\ge 1}$, $\lambda\in\T$, $p(z)=\sum_{j=0}^{m-1}c_jz^j$ with $p(\lambda)\ne 0$, and $\vec{\mu}=(\frac{|dz|}{2\pi},\mu_1,\ldots,\mu_{2m-1})$ is given by \eqref{eq:H(b)=Dvecmu}, then $\vec{\mu}$ is a normalized allowable $2m$-tuple, and there exists $b\in\HS_0(\C,\C)$ such that $\HH(b)=\HD_{\vec{\mu}}$.
\end{theorem}

As a final remark we point out that the notion of allowable tuples is far from fully understood. Theorem~\ref{theorem:H(b)=Dvecmu} and related results provide many new examples of allowable tuples.

The paper is organized as follows: In Section~\ref{section:Preliminaries} we introduce some notation and preliminary results. In Section~\ref{section:DirichletTypeH[B]Shifts} we recall some properties of $\HD(\mu)$ and use these to study Dirichlet-type shifts on $\HH[B]$-spaces. We will prove Theorems \ref{theorem:HBDmuequal} and \ref{dirichlettype} in this section. We also outline a method for calculating the reproducing kernel of $\HD(\mu)$ for finitely atomic $\mu$. In Section~\ref{section:n-isometricH[B]Shifts} we recall some properties of $\HD_{\vec{\mu}}$ and use these to study expansive $n$-isometric shifts in $\HH[B]$-spaces with finite rank. We will prove Theorem \ref{theorem:H(b)=Dvecmu} and then answer some questions left open in \cite{Ry19} in this section.

\section{Preliminaries}\label{section:Preliminaries}

\subsection{Notation and terminology} We will use the basic notation $\D=\{z\in\C\mid |z|<1\}$, $\T=\{z\in\C\mid |z|=1\}$, $\overline{\D}=\D\cup\T$, and $\Z_{\ge x}=\Z\cap[x,\infty)$. Integration with respect to area measure on $\D$ and arc length measure on $\T$ are denoted by $dA$ and $|dz|$ respectively. We use $\delta_\lambda$ to denote the Dirac measure for $\lambda\in \C$. The standard Kronecker delta is given by $\delta_{a,b}=\delta_a(\{b\})$.

Given a set $E\subseteq\HH$, we let $\bigvee E$ denote the closed linear span of $E$, and $[E]_T$ the smallest closed $T$-invariant subspace containing $E$. For singletons $E=\{e\}$, we write $[e]_T$ rather than $[E]_T$. If $[e]_T=\HH$ for some vector $e\in\HH$, then we say that $T$ is \emph{cyclic} and that $e$ is a \emph{cyclic vector} for $T$.

Given an analytic function $f\colon \D\to\C$, we let $f^{(n)}$ denote the $n$th derivative of $f$. We also define the corresponding forward shift
\begin{align*}
	M_z f\colon z\mapsto zf(z),
\end{align*}
and backward shift
\begin{align*}
	Lf\colon z\mapsto \frac{f(z)-f(0)}{z}.
\end{align*}
We write $(M_z,\mathcal{X})$ or $(L,\mathcal{X})$ to denote the restriction of these operators to some space $\mathcal{X}$ of analytic functions. We sometimes use the same notation to denote the corresponding operations acting on vector-valued functions.

Given a set $X$, and two functions $f,g\colon X\to [0,\infty]$, we write $f(x)\lesssim g(x)$ to indicate the existence of $C\in[0,\infty)$ such that $f(x)\le C g(x)$ whenever $x\in X$ and $g(x)<\infty$. If $f(x)\lesssim g(x)$ and $g(x)\lesssim f(x)$, then we write $f(x)\approx g(x)$.

\subsection{The Schur class $\HS(\ell^2,\C)$} Let $B\colon\D\to \ell^2$ be analytic. If $\|B(z)\|_{\ell^2}\le 1$ for each $z\in\D$, then we say that $B$ is a \emph{Schur function}. The class of Schur functions is denoted $\HS(\ell^2,\C)$. We will be especially interested in Schur functions for which $B(0)=0$. This condition corresponds to a type of normalization. The class of normalized Schur functions is denoted $\HS_0(\ell^2,\C)$.

We equip $\HS(\ell^2,\C)$ with the equivalence relation $\sim$ given by $B_1\sim B_2$ if and only if $\La B_1(z),B_1(\lambda)\Ra_{\ell^2}=\La B_2(z),B_2(\lambda)\Ra_{\ell^2}$ for all $z,\lambda\in\D$. If the equivalence class $[B]$ has a representative $(b_k)_{k=1}^\infty$ where only finitely many $b_k$ are non-zero, then we say that $B$ has \emph{finite rank}. If $B$ has finite rank and $n$ is the minimal number of non-zero $b_k$ among all representations of $[B]$, then we say that $\rank B = n$. If $B$ does not have finite rank, then we say that $\rank B = \infty$. If $B\in\HS(\ell^2,\C)$ and $\rank B = n <\infty$, then we sometimes write $B\in\HS(\C^n,\C)$, and allow ourselves the somewhat sloppy notation $B=(b_1,\ldots,b_n)$. Note that the particular case $\HS(\C,\C)$ is just the unit ball of the Hardy algebra $H^\infty$.

Let $H^2$ be the Hardy space. Let $B\in\HS(\ell^2,\C)$. Suppose that $a\in H^2$ satisfies $|a(z)|^2+\|B(z)\|_{\ell^2}^2=1$ for $|dz|$-a.e. $z\in\T$. Such an $a$ is called a \emph{mate} of $B$. The existence of a mate is equivalent to the condition
\begin{align}\label{eq:MateExistenceCondition}
	\int_{\T} \log\left(1-\|B(z)\|_{\ell^2}^2\right)\,\frac{|dz|}{2\pi}>-\infty,
\end{align}
see for instance \cite[Chapter~II.4]{Gar}. In fact, the above condition guarantees the existence of an \emph{outer} mate function $a$. Such functions are unique up to multiplicative factors $c\in\T$. In particular, if $a$ is a mate of $B$, if $a$ is outer, and if $a(0)>0$, then $a$ is uniquely determined, and we call $a$ \emph{the mate} of $B$. In the case of rank one Schur functions, $B\in\HS(\C,\C)$ has a mate if and only if $B$ is a non-extreme point in the unit ball of $H^\infty$.

Lastly, suppose $B\in\HS(\C^n,\C)$, and $B=\frac{1}{q}(p_1,\ldots,p_n)$ for some polynomials $p_1,\ldots,p_n,q$. We then say that $B$ is \emph{rational}. For such a representation of $B$, define
\begin{align*}
	N:=\max_{p\in\{p_1,\ldots,p_n,q\}}\deg p.
\end{align*}
If $N$ is minimal over all such representations, then we say that $B$ has \emph{degree $N$}, $\deg B=N$. If $B$ is not rational, then we say that $\deg B = \infty$.

\subsection{de~Branges--Rovnyak spaces $\HH[B]$} For $B\in\HS_0(\ell^2,\C)$ we define
\begin{align*}
	K^B_w(z)=\frac{1-\La B(z),B(w)\Ra_{\ell^2}}{1-z\overline w},\quad z,w\in\D.
\end{align*}
The function $K^B\colon (z,w)\mapsto K^B_w(z)$ is positive definite, so $K^B$ is the reproducing kernel of a reproducing kernel Hilbert space (RKHS). The RKHS is uniquely determined by $K^B$. However, $K^B$ does not determine $B$, but rather the equivalence class $[B]$. To emphasize this, the RKHS with reproducing kernel $K^B$ is denoted $\HH[B]$. The exception is when $B\sim (b,0,0,\ldots)$, i.e. $\rank B = 1$, in which we write $\HH(b)$. We call $\HH[B]$ the de~Branges--Rovnyak space.

$\HH[B]$ spaces have a rich history originating from~\cite{dBR65}. We refer to \cite{BallBoloBasics, dBR66, FM16, FM162, Sa94} for more general background. A characteristic property of the spaces $\HH[B]$ is that they are RKHS in which the backward shift is a contraction. We state a special case of \cite[Proposition~2.1]{AM19}.

\begin{proposition}[\cite{AM19}]\label{proposition:HisHB}
	Let $\HH$ be a RKHS. The following are equivalent.
	\begin{enumerate}
		\item There exists $B\in\HS_0(\ell^2,\C)$ such that $K^B$ is the reproducing kernel for $\HH$.
		\item $\HH$ is closed under the backward shift $L$, and $(L,\HH)$ is a contraction. Moreover, the constant function $1\in\HH$, and $f\in\HH$ implies $f(0)=\La f,1\Ra_\HH$.
	\end{enumerate}
\end{proposition}

We henceforth use $\HH[B]$ as generic notation for any RKHS satisfying the equivalent conditions of Proposition~\ref{proposition:HisHB}. Indeed, since a RKHS is uniquely determined by its reproducing kernel, any RKHS satisfying the above conditions is \emph{equal} to $\HH[B]$ for some $B\in\HS_0(\ell^2,\C)$.

The conditions in Proposition~\ref{proposition:HisHB} do not guarantee that $\HH[B]$ is invariant under $M_z$. If on the other hand $M_z\HH[B]\subseteq\HH[B]$, then $(M_z,\HH[B])$ is bounded by the closed graph theorem. If $\rank B<\infty$, then a result by Aleman and Malman \cite[Theorem~5.2]{AM19} states that $M_z\HH[B]\subseteq\HH[B]$ if and only if \eqref{eq:MateExistenceCondition} holds. Another formulation of this is that $M_z$ defines a bounded operator on $\HH[B]$ if and only if $B$ has a mate function.

Assume now that $M_z\HH[B]\subseteq\HH[B]$. Since $(L,\HH[B])$ is a contraction,
\begin{align*}
	\|f\|_{\HH[B]}^2=\|LM_zf\|_{\HH[B]}^2\le\|M_zf\|_{\HH[B]}^2,
\end{align*}
so $(M_z,\HH[B])$ is expansive. The following result by Gu, Richter, and the first author~\cite{LGR} puts $\HH[B]$-spaces in the context of general bounded analytic expansive operators.
\begin{proposition}[\cite{LGR}]\label{proposition:HequivalenttoHB}
	Let $T\in\HB(\HH)$ with $\dim\ker T^*=1$. The following are equivalent.
	\begin{enumerate}
		\item $T$ is analytic and expansive.
		\item There exists $B\in\HS_0(\ell^2,\C)$ such that $M_z\HH[B]\subseteq\HH[B]$ and $T$ is unitarily equivalent to $(M_z,\HH[B])$.
	\end{enumerate}
\end{proposition}
	Similar to Proposition~\ref{proposition:HisHB}, we have stated a special case of a more general result, in this case \cite[Theorem~4.6]{LGR}. The restriction $\dim\ker T^*=1$ can be removed by considering vector-valued analogues of $\HH[B]$.

\section{Dirichlet-type shifts on finite rank $\HH[B]$-spaces.}\label{section:DirichletTypeH[B]Shifts}



\subsection{Weighted Dirichlet spaces $\HD(\mu)$} Let $f\in H^2$ and $\lambda\in\overline{\D}$. If $f(\lambda)$ exists (in the sense of a radial limit), then we define the corresponding \emph{local Dirichlet integral} by
\begin{align*}
	D_\lambda(f)=\int_{z\in\T}\left|\frac{f(z)-f(\lambda)}{z-\lambda}\right|^2\frac{|dz|}{2\pi}.
\end{align*}
Guided by for instance \cite[Section~7.1, and Theorem~7.2.5]{EKMR14}, we adopt the convention that $D_\lambda(f)=\infty$ if $f(\lambda)$ does not exist.

Let $\mu$ be a  finite positive Borel measure on $\overline{\D}$. We define the Hilbert space
\begin{align*}
	\HD(\mu)=\{f\in H^2\mid \int_{\lambda\in\overline{\D}}D_\lambda(f)\, d\mu(\lambda)<\infty\},
\end{align*}
with norm given by
\begin{align*}
	\|f\|_{\HD(\mu)}^2=\|f\|_{H^2}^2+\int_{\lambda\in\overline{\D}}D_\lambda(f)\, d\mu(\lambda).
\end{align*}
The measure $\mu$ can be uniquely identified with the positive superharmonic function
\begin{align*}
	U_\mu\colon z \mapsto \int_{\D}\log\left|\frac{1-\zeta\bar z}{\zeta-z}\right|^2\, \frac{d\mu(\zeta)}{1-|\zeta|^2}
	+
	\int_{\T}\frac{1-|z|^2}{|\zeta-z|^2}\, d\mu(\zeta),\quad z\in\D,
\end{align*}
for which it holds that
\begin{align*}
	\int_{\lambda\in\overline{\D}}D_\lambda(f)\, d\mu(\lambda)=\int_\D |f'(z)|^2U_\mu(z)\, dA(z).
\end{align*}
Moreover, any positive superharmonic function may be obtained in this way. Hence the term \emph{weighted Dirichlet space}.

The special case where $\mu=\delta_\lambda$ is called a \emph{local Dirichlet space}, and is denoted by $\HD_\lambda^1$. For measures supported on $\T$, the space $\HD(\mu)$ space was introduced in \cite{Ri91}. The general case was thoroughly investigated in \cite{Al93}. The theory of $\HD(\mu)$ spaces has received a lot of attention in the recent years, and we refer to \cite{Al93, ARSW19, EKMR14} for more background on the subject.


The fact that $(M_z,\HD(\mu))$ is of Dirichlet type relies in part on the useful identity
\begin{align}\label{eq:DirichletMzIdentity}
	\|M_zf\|_{\HD(\mu)}^2=\|f\|_{\HD(\mu)}^2 + \int_{\overline{\D}}|f(z)|^2\, d\mu(z),
\end{align}
see for example \cite[(5.2)]{LR15}.


Since $H^2$ is a RKHS and $\HD(\mu)$ is contractively contained in $H^2$, $\HD(\mu)$ is also a RKHS. Even though this does not yield an explicit formula for the kernel, it is plain that any $f\in\HD(\mu)$ satisfies
\begin{align*}
	\La f,1\Ra_{\HD(\mu)}=\La f,1\Ra_{H^2}=f(0).
\end{align*}
By Proposition~\ref{proposition:HisHB}, we conclude that there exists $B\in\HS_0(\ell^2,\C)$ such that $\HD(\mu)=\HH[B]$.

\subsection{The finite rank case}

The first main result of this section is Theorem \ref{theorem:HBDmuequal} which is about the relation between $\HH[B]$ with finite rank and $\HD(\mu)$ for finitely atomic measures.
As preparation for the proof of this theorem, we begin with a simple necessary condition for when an $\HH[B]$-space is also a $\HD(\mu)$-space.

\begin{lemma}\label{proposition:FiniteRankDmuIsRational}
Let $B\in\HS_0(\ell^2,\C)$, and suppose that $\HH[B]=\HD(\mu)$ for some $\mu$. Then $\rank B = \rank \Delta = \deg B$.
\end{lemma}
\begin{proof}
It is known that for an $\HH[B]$ space we have $\rank\Delta=\rank B$, see e.g. \cite[Theorem 4.6 and Lemma 5.1]{LGR}. If $\rank B = \infty$, then $\deg B = \infty$ by definition. It remains to prove that $\deg B=\rank B$ whenever the latter is finite.

Assuming $B=(b_j)_{1\le j \le n}$, $M_z$ is bounded on $\HH[B]$ if and only if \eqref{eq:MateExistenceCondition} holds, see \cite[Theorem~5.2]{AM19}. Since $M_z$ is bounded on $\HD(\mu)$, we conclude from \cite[Theorem~1.2]{LGR} that $\deg B = \dim[\ran\Delta]_{T^*}$. To finish the proof it is now sufficient to show that $[\ran\Delta]_{T^*}=\overline{\ran\Delta}$. The fact that $\overline{\ran\Delta}$ is $T^*$-invariant is implied by the inequality
\begin{align*}
	T^{*2}T^2-2T^*T +I = T^*\Delta T - \Delta \leq 0,
\end{align*}
which is immediate from \eqref{eq:DirichletMzIdentity}.
\end{proof}

The following lemma can be derived from the proof of \cite[Theorem 7.2]{LGR}. This lemma plays an important role for our investigation of the relation between $\mu$ and $B$, we include a proof here for completeness.
\begin{lemma}\label{lemma:MateFunctionAndPhi}
Let $B=(b_1,\dots,b_k)$ be a rational Schur function of degree $n$ such that $(M_z,\HH[B])$ is a bounded operator and $B(0)=0$. Let $T=(M_z,\HH[B])$ and $\HN=[\ran \Delta]_{T^*}$. Then $\HN$ is $L$-invariant and $\dim \HN = n$. Let $\HM = \HH[B] \ominus \HN$, $\varphi \in \HM \ominus z \HM, \|\varphi\| = 1$. Let $a$ be the mate of $B$.

Suppose $\prod_{i=1}^n(z-\alpha_i)$ is the characteristic polynomial of $L|\HN$, and $\prod_{i=1}^n (z-\overline{\lambda_i})$ is the characteristic polynomial of $T^*|\HN$. Then
\begin{itemize}
\item[(i)]
$$a(z) = a(0)\frac{\prod_{i=1}^n (1-\overline{\lambda_i}z)}{\prod_{i=1}^n(1-\alpha_i z)},$$
\item[(ii)] $\HH[B] = \varphi H^2 \oplus \HN$, $ |\varphi(z)| = |a(z)|$ on $\T$ and
$$\varphi(z) = e^{it} a(0)\frac{\prod_{i=1}^n (z-\lambda_i)}{\prod_{i=1}^n(1-\alpha_i z)},t\in \R.$$
\end{itemize}
\end{lemma}
\begin{proof}
Note that by \cite[Lemma 6.1 and Theorem 6.2]{LGR}, $\HN=[\ran \Delta]_{T^*} = [\ran \Delta]_{L}$, and $\dim \HN = \deg B = n$. Then by \cite[(7.1)]{LGR},
\begin{align}\label{N as rational functions}
\HN=\left\{\frac{p(z)}{\prod_{i=1}^n(1-\alpha_i z)}: p \text{ is a polynomial of degree }<n\right\},
\end{align}
and $\prod_{i=1}^n(1-\alpha_i z)$ is a constant multiple of the lowest common denominator of the $b_i$'s.
It is clear that $T|\HM$ is an isometry. Since $\dim (\HH[B]\ominus z\HH[B]) = 1$, $\dim \HN = n < + \infty$, we have $\dim (\HM \ominus z \HM) = 1$, see e.g. \cite[Lemma 2.1]{ARS02}. So $T|\HM$ is unitarily equivalent to $(M_z, H^2)$. Thus $\HM = \varphi H^2$, $P_{\HM}K^B_w(z)=\frac{ \varphi(z)\overline{\varphi(w)}}{1-z\overline{w}}$ and $\HH[B] = \varphi H^2 \oplus \HN$. Since $K^B_w(z)=P_{\HN}K^B_w(z)+ P_{\HM}K^B_w(z)$ we conclude that
$$1-\sum_{i=1}^k|b_i(z)|^2= (1-|z|^2)\|P_\HN K^B_z\|^2 + |\varphi(z)|^2$$ for all $z\in \D$. Since all functions in $\HN$ are bounded and $\HN$ is finite dimensional we let $|z|\to 1$ and obtain that $|\varphi|^2=|a|^2$ on $\T$.

 Let $C= (I-P_{\HN})M_zP_{\HN}, A = P_{\HN}M_zP_{\HN}$. Then
$$T=\left[\begin{matrix}T|\HM&C\\0&A\end{matrix}\right].$$
It follows from
$$T^*T-I= \left[\begin{matrix}0&(T|\HM)^*C\\C^*T|\HM &C^*C+A^*A-I\end{matrix}\right]\ge 0$$
that $(T|\HM)^*C = 0$. Note that $\ker (T|\HM)^* = \C \varphi$. Thus $\rank C = 1$ and there exists $f_1 \in \HN$ such that $C f_1 = \varphi$. Then there exists $f_2 \in \HN$ such that $\varphi = zf_2 - f_1$. Thus by (\ref{N as rational functions}), $\varphi(z)= \frac{h(z)}{\prod_{i=1}^n(1-\alpha_i z)}$ for some polynomial $h$ of degree $\le n$.

Since $\HH[B] = \HM \oplus \HN, \HM = \varphi H^2$, we have for any $g \in \HN \subseteq H^2$, $\varphi g \in \HM$. So $h g \in \HM$ or $P_{\HN}h(M_z)|\HN=0$. Let $\tilde{h}(z)=\overline{h(\overline{z})}$, then $\tilde{h}(M_z^*)|\HN=0$. Since by \cite[Lemma 7.1]{LGR}, the minimal polynomial of $T^*|\HN$ equals the characteristic polynomial of $T^*|\HN$ and it has degree $n$, we conclude that $\tilde{h}$ must be a multiple of $\prod_{i=1}^n (z-\overline{\lambda_i})$. This implies that $\varphi(z) = \gamma \frac{\prod_{i=1}^n (z - \lambda_i)}{\prod_{i=1}^n(1-\alpha_i z)}$ for some $\gamma\in \C$. Now $a$ is outer, but has the same modulus as $\varphi$ on $\T$, hence $a(z)= a(0)\frac{\prod_{i=1}^n (1-\overline{\lambda_i}z)}{\prod_{i=1}^n(1-\alpha_i z)}$. The proof is complete.
\end{proof}

We remark here that in the setting of Lemma~\ref{lemma:MateFunctionAndPhi}, the functions in $\HH(B)$ have nontangential limits at $\lambda_i, i =1, 2, \ldots, n$ and $\HN$ is spanned by the reproducing kernels $K_{\lambda_i}^B$ or the derivatives of $K_{\lambda_i}^B$ when the $\lambda_i$ has multiplicities $\geq 2$.

Now we can prove Theorem~\ref{theorem:HBDmuequal}.


\begin{proof}[Proof of Theorem~\ref{theorem:HBDmuequal}]
Suppose $\HH[B]=\HD(\mu)$ for some $\mu = \sum_{j=1}^n c_j \delta_{\lambda_j}$. By \eqref{eq:DirichletMzIdentity},
\begin{align*}
	\La \Delta f,f\Ra_{\HD(\mu)}=\sum_{j=1}^nc_j|f(\lambda_j)|^2,
\end{align*}
i.e. $\Delta = \sum_{i=1}^n c_i K_{\lambda_i}^{B}\otimes K_{\lambda_i}^{B}$. Hence, $\rank \Delta = n$ and $(i)$ implies $(ii)$. By Lemma~\ref{proposition:FiniteRankDmuIsRational}, $B$ is rational with $\deg B = n$, so $(ii)$ implies $(iii)$.

For the final implication, suppose $B$ is rational with degree $n$. Let $\HN = [\ran \Delta]_{T^*}$, and $\HM = \HH[B] \ominus \HN$. By Lemma~\ref{proposition:FiniteRankDmuIsRational}, $\dim \HN = \deg B = \rank \Delta = n$. Let $\varphi \in \HM \ominus z \HM$, $\|\varphi\| = 1$. By Lemma~\ref{lemma:MateFunctionAndPhi} we have that $\HH[B] = \varphi H^2 \oplus \HN$, and $\varphi(z) = \frac{\prod_{i=1}^n (z-\lambda_i)}{q(z)}$ for some polynomial $q(z)$ and some $\lambda_1, \ldots \lambda_n \in \overline{\D}$. Since $\varphi \in \left(\ran\Delta\right)^\perp$, \eqref{eq:DirichletMzIdentity} implies that
\begin{align*}
	0 = \langle\Delta \varphi, \varphi\rangle = \int_{\overline{\D}}|\varphi(z)|^2d\mu(z).
\end{align*}
Thus $|\varphi|^2 = 0$ $\mu$-a.e. Since $|\varphi(z)|^2 > 0$ on $\overline{\D} \backslash \{\lambda_1,\ldots,\lambda_n\}$, we obtain that $\mu=\sum_{j=1}^nc_j\delta_{\lambda_j}$. To see that $c_1,\ldots,c_n>0$ and $\lambda_1,\ldots,\lambda_n\in\overline{\D}$ are distinct in this case, we note that if this is not so, then $\rank \Delta =\deg B$ would be strictly less than $n$.
\end{proof}

If $B\in\HS_0(\ell^2,\C)$ is a rational Schur function such that $M_z\HH[B]\subseteq\HH[B]$, then the polynomials are dense in $\HH[B]$ (\cite[Theorem 5.5]{AM19}). When $B\in\HS_0(\ell^2,\C)$ is such that $M_z\HH[B]\subseteq\HH[B]$ and polynomials are dense in $\HH[B]$, it was proved in \cite[Lemma 3.6]{CGL} that for $\lambda \in \overline{\D}$, $K_\lambda^B$ exists if and only if $\overline{\lambda}$ is an eigenvalue of $M_z^*$, which is equivalent to the condition that every function in $\HH(B)$ has a non-tangential limit at $\lambda$.

Now we prove Theorem \ref{dirichlettype}.

\begin{proof}[Proof of Theorem~\ref{dirichlettype}]
Suppose $\mu = \sum_{i=1}^\infty c_i \delta_{\lambda_i}$. By \eqref{eq:DirichletMzIdentity},
\begin{align}\label{discreteequ}
\langle\Delta f, f \rangle = \int_{\overline{\D}}|f(z)|^2d\mu(z) = \sum_{i=1}^\infty c_i |f(\lambda_i)|^2, \quad f \in D(\mu).
\end{align}
We conclude that $\Delta = \sum_{i=1}^\infty c_i K_{\lambda_i}^{B}\otimes K_{\lambda_i}^{B}$, i.e. $(i)\implies(ii)$.

For the reverse implication, we use that $T^* K_{\lambda_i}^B = \overline{\lambda_i}K_{\lambda_i}^B$, so
\begin{align*}
	T^* \Delta T = \sum_{i=1}^\infty c_i T^* K_{\lambda_i}^B\otimes T^*K_{\lambda_i}^B = \sum_{i=1}^\infty c_i |\lambda_i|^2 K_{\lambda_i}^B\otimes K_{\lambda_i}^B.
\end{align*}
Thus for $n \geq 2$,
\begin{align*}
\sum_{j=0}^n (-1)^j \binom{n}{j} T^{*j} T^j &= \sum_{j=0}^{n-1} (-1)^j \binom{n-1}{j} (T^{*j}T^j - T^{*(j+1)}T^{(j+1)})\\
&= \sum_{j=0}^{n-1}(-1)^j \binom{n-1}{j} T^{*j}(-\Delta) T^j\\
& = -\sum_{j=0}^{n-1}(-1)^j \binom{n-1}{j} \sum_{i=1}^\infty c_i |\lambda_i|^{2j} K_{\lambda_i}^B\otimes K_{\lambda_i}^B\\
& = - \sum_{i=1}^\infty c_i(1-|\lambda_i|^2)^{n-1} K_{\lambda_i}^B\otimes K_{\lambda_i}^B \leq 0.
\end{align*}
It follows that $T$ is of Dirichlet-type. By \cite[Theorem~IV.2.5]{Al93} there exists a finite positive measure $\mu$ and a unitary operator $U\colon \HH[B] \rightarrow \HD(\mu)$ that intertwines $T$ with $T_\mu = (M_z,\HD(\mu))$, i.e. $UT=T_\mu U$. The intertwining relation implies that $U^*$ maps $\ker T_\mu^*$ onto $\ker T^*$. Since $\ker T_\mu^*$ consists precisely of the constant functions on $\D$, we may conclude that $U1 = c$ for some constant $c\in\T$. Since polynomials are dense in $\HD(\mu)$, we conclude that $Uf=cf$ for all $f\in\HH[B]$. In other words, $\HH[B] = \HD(\mu)$. Since $\Delta = \sum_{i=1}^\infty c_i K_{\lambda_i}^B\otimes K_{\lambda_i}^B$, (\ref{discreteequ}) implies that $\mu = \sum_{i=1}^\infty c_i \delta_{\lambda_i}$.
\end{proof}

\subsection{Construction for the reproducing kernel}Now we discuss how to find the rational Schur function $B$ such that $\HD(\mu)=\HH[B]$ when $\mu$ is a finitely atomic measure. Suppose $\HD(\mu) = \HH[B]$ with $\mu = \sum_{i=1}^n c_i \delta_{\lambda_i}$, $c_1,\ldots,c_n > 0$, $\lambda_1,\ldots, \lambda_n \in \overline{\D}$. If all the $\lambda_i$ are in $\T$, then $T=(M_z,\HD(\mu))$ is a $2$-isometry. In this case, it is explained in \cite[Section 11]{LGR}, see also \cite{Sa98}, \cite{Co16}, \cite[Theorem 6.4]{CGR22}. If $\mu$ is supported at one point, i.e. $n=1$, then it is explained in \cite{Sa97, CGR10, CR13, EFKKMR16}.
\begin{lemma}\label{calcuequal}
Consider a finitely atomic measure $\mu=\sum_{i=1}^n c_i \delta_{\lambda_i}$, where $c_1,\ldots, c_n > 0$ and $\lambda_1,\ldots,\lambda_n \in \overline{\D}$. For any such measure, there exists a rational Schur function $B\in\HS_0(\ell^2,\C)$ with $\rank B = \deg B = n$, and $\HD(\mu) = \HH[B]$. Furthermore, $B\sim \frac{1}{q}(p_1,\ldots,p_n)$, where $q$ and $p_1,\ldots,p_n$ satisfy the following relations,
\begin{align*}
|q(z)|^2 = \prod_{i=1}^n |z-\lambda_i|^2 +\sum_{i=1}^n c_i \prod_{j\neq i} |z-\lambda_j|^2,\quad  z \in \T,
\end{align*}
and
\begin{align*}
	p_i(z) = \langle PX(z), e_i\rangle_{\C^n},\quad  i = 1, \ldots, n.
\end{align*}
Here $P$ is an upper triangular $n \times n$ matrix, $X(z) = (z^j)_{1\le j\le n}$, and $\{e_i\}$ is the standard orthonormal basis of $\C^n$.
\end{lemma}
\begin{proof}
The norm in $\HD(\mu)=\HH[B]$ is
\begin{align}\label{normdu}
\|f\|^2 = \|f\|_{H^2}^2 + \sum_{i=1}^n c_i D_{\lambda_i}(f).
\end{align}
Let $\HN = [\ran \Delta]_{T^*}, \HM = \HH(B) \ominus \HN$. Suppose $B = (b_1, \ldots, b_n) = \frac{1}{q}(p_1,\ldots,p_n)$, where $p_i$ are polynomials with $\deg p_i \leq n$, and $q$ is a constant multiple of the least common denominator of the $b_i$'s with $q(0) = 1/a(0)$, $a$ being the mate of $B$. Since $\Delta = \sum_{i=1}^n c_i K_{\lambda_i}^{\mu}\otimes K_{\lambda_i}^{\mu}$, we see that $\HN = \bigvee\{K_{\lambda_i}^{\mu}=K_{\lambda_i}^{B}\mid i = 1, \ldots, n\}$. Let $\varphi \in \HM \ominus z\HM, \|\varphi\| = 1$. Then by Lemma~\ref{lemma:MateFunctionAndPhi} we have $\HH[B] = \varphi H^2 \oplus \HN$, $\varphi(z) = \frac{\prod_{i=1}^n (z-\lambda_i)}{q(z)}$ and $a(z) = \frac{\prod_{i=1}^n(1-\overline{\lambda_i}z)}{q(z)}$.
Now using $\varphi(\lambda_i) = 0$, $i =1, \ldots,n$ we obtain for $k \geq 1$
\begin{align*}
0 = \langle z^k\varphi,\varphi\rangle &= \langle z^k\varphi,\varphi\rangle_{H^2}+\sum_{i=1}^n c_i \int_\T \frac{z^k\varphi(z) \overline{\varphi(z)}}{|z-\lambda_i|^2}\frac{|dz|}{2\pi}\\
& = \int_\T z^k |a(z)|^2 \left(1+ \sum_{i=1}^n \frac{c_i}{|z-\lambda_i|^2}\right) \frac{|dz|}{2\pi}.
\end{align*}
Thus $|a(z)|^2 \left(1+ \sum_{i=1}^n \frac{c_i}{|z-\lambda_i|^2}\right) = 1, a.e. ~\T$, and so
\begin{align*}
|q(z)|^2 = \prod_{i=1}^k |z-\lambda_i|^2 +\sum_{i=1}^k c_i \prod_{j\neq i} |z-\lambda_j|^2, \quad z \in \T.
\end{align*}
Since the $p_i(z)$ are polynomials, we can assume $p_i(z) = \langle PX(z), e_i\rangle_{\C^n}$, $i = 1, \ldots, n$ for some $n \times n$ matrix $P$. Note that
\begin{align*}
B(z) B(w)^*& = \frac{1}{q(z)\overline{q(w)}}\sum_{i=1}^n \langle PX(z), e_i\rangle_{\C^n} \overline{\langle PX(w), e_i\rangle_{\C^n}}\\
& = \frac{1}{q(z)\overline{q(w)}} \langle PX(z), PX(w)\rangle_{\C^n}\\
& = \frac{1}{q(z)\overline{q(w)}} \langle P^*PX(z), X(w)\rangle_{\C^n}.
\end{align*}
Since $P^*P$ is positive semi-definite, by Cholesky's theorem \cite[Theorem~4.4]{PR16}, we can take $P$ to be an upper triangular matrix. The proof is complete.
\end{proof}
In the setting of the above lemma,
\begin{align*}
K^{\mu}_w(z) = K^B_w(z)= P_{\HN}K_w^B(z)+\frac{\varphi(z)\overline{\varphi(w)}}{1-z\overline{w}}.
\end{align*}
If $\varphi(z) = \frac{\prod_{i=1}^n (z-\lambda_i)}{q(z)}$ is determined, then to obtain $K^B_w(z)$, it is left to determine $P_{\HN}K_w^B(z)$, which can be done by using the dual basis of $\{K_{\lambda_i}^{B}\}_{i=1}^n$.

To be precise, suppose $\{f_i\}_{i=1}^n$ is a dual basis of $\{K_{\lambda_i}^{B}\}_{i=1}^n$, i.e.
\[
\langle f_i, K_{\lambda_j}^{B}\rangle = \delta_{ij},\quad 1\le i,j\le n.
\]
Since the functions in $\HN$ have the form $\frac{p(z)}{q(z)}$ with $p$ polynomials of degree $< n$ (\cite[(7.1)]{LGR}), we get that for each $i$, $f_i(z) = d_i \frac{\prod_{j\neq i} (z-\lambda_j)}{q(z)}$, where $d_i$ is determined by $f_i(\lambda_i) = 1$.
Note that
\[
P_\HN = \sum_{i=1}^n K_{\lambda_i}^B \otimes f_i = \sum_{i=1}^n f_i\otimes K_{\lambda_i}^B.
\]
Then $P_{\HN}K_w^B(z) = \sum_{i=1}^n f_i(z) \overline{K_{\lambda_i}^B(w)}$. Also for each $j$, $f_j = \sum_{i=1}^n K_{\lambda_i}^B \langle f_j, f_i\rangle$, where $\langle f_j, f_i\rangle$ can be calculated by (\ref{normdu}). Finally we solve for $K_{\lambda_i}^B$ from the linear equations $f_j = \sum_{i=1}^n K_{\lambda_i}^B \langle f_j, f_i\rangle$ , $j =1, \ldots, n$, and obtain $P_{\HN}K_w^B(z)$.

We use the following example to illustrate the above construction.
\begin{example}\label{reproducingtwopts}
Let $\mu = \delta_1 + \delta_{0}$. Then the norm in $\HD(\mu)$ is
\[
\|f\|^2 = \|f\|_{H^2}^2 + D_1(f) + D_{0}(f).
\]
We have $\HD(\mu) = \HH[B]$ for some rational Schur function $B$. Note that $K_0^B(z) = K_0^\mu(z) = 1$, $z \in \D$. Using the notation as above, we have $\Delta = K_1^B \otimes K_1^B + 1 \otimes 1$, $\varphi(z) = \frac{(z-1)z}{q(z)}$, and $a(z) = \frac{1-z}{q(z)}$, where $q(0) > 0$ and
\[
|q(z)|^2 = |z-1|^2|z|^2 + |z|^2 + |z-1|^2, \quad z \in \T.
\]
So $q(z) = 2-z$. Let $\{f_1, f_2\}$ be the dual basis of $\{K_1^B, 1\}$. Then
\[
f_1(z) = \frac{z}{2-z},\quad f_2(z) = \frac{2(1-z)}{2-z}.
\]
By calculation, we have $\|f_1\|^2 = 2, \langle f_1, f_2\rangle = 2, \|f_2\|^2 = 3$. We then obtain
\begin{align*}
K_1^B (z) = \frac{-z/2 +2}{2-z}.
\end{align*}
Thus
\begin{align*}
K^B_w(z)& = P_{\HN}K_w^B(z)+\frac{\varphi(z)\overline{\varphi(w)}}{1-z\overline{w}}\\
& = f_1(z) \overline{K_{1}^B(w)} + f_2(z) + \frac{1}{1-z\overline{w}} \frac{(z-1)z}{q(z)} \frac{\overline{(w-1)w}}{\overline{q(w)}}\\
& = \frac{1-(z\overline{w})(1/2z\overline{w} -z-\overline{w}+ 5/2)/((2-z)(2-\overline{w}))}{1-z\overline{w}}.
\end{align*}
For the Schur function $B$, we can take $B(z) = (\frac{2z^2/\sqrt{10}-\sqrt{10}z/2}{2-z}, \frac{z^2/\sqrt{10}}{2-z})$.
\end{example}

\section{$n$-isometric shifts on finite rank $\HH[B]$-spaces}\label{section:n-isometricH[B]Shifts}

\subsection{The $\HD_{\vec{\mu}}$-model for cyclic $n$-isometries} We essentially follow the notation in \cite{Ry19}, but deviate by using $\Delta^{(n)}$ rather than $\beta_n(T)$ to denote the higher order defect operators. Let $\HD$ be the space of $C^\infty$ functions on $\T$, and $\hat{f}(k) = \int_\T f(\zeta) \overline{\zeta}^k \frac{|d\zeta|}{2\pi}$ denote the standard Fourier coefficients of $f\in\HD$. Of particular interest will be the Poisson kernels associated to $z\in\D$, i.e. the smooth functions
\begin{align*}
	P_z\colon\zeta\mapsto = \frac{1-|z|^2}{|z-\zeta|^2} = \sum_{k=0}^\infty (\overline{\zeta}z)^k + \sum_{k=1}^\infty (\zeta\overline{z})^k,\quad \zeta\in\T.
\end{align*}

Let $\HD'$ be the dual of $\HD$, the space of distributions on $\T$, and $\hat{\mu}(k) = \mu(\overline{\zeta}^k)$ denote the standard Fourier coefficients of $\mu\in\HD'$. With these conventions the duality pairing becomes
\begin{align*}
	\La f,\mu\Ra = \sum_{k\in\Z}\hat f(k)\hat \mu(-k),\quad f\in\HD,\mu\in\HD'.
\end{align*}
The Poisson extension of $\mu\in\HD'$ is the harmonic function
\begin{align*}
	P_\mu\colon z\mapsto \mu(P_z)=\sum_{k=0}^\infty \hat{\mu}(k) z^k + \sum_{k=1}^\infty \hat{\mu}(-k) \overline{z}^k,\quad z\in\D.
\end{align*}
With the standard identification of $\HD$ as a subspace of $\HD'$, we say that $f\in\HD$ is analytic if its Poisson extension is analytic. This happens precisely when $\hat f (k)=0$ for $k<0$, and we denote the corresponding class of analytic functions by $\HD_a$.

Given $\mu \in \HD'$ and $f \in \HD_a$, the harmonically weighted Dirichlet integral of order $i \in \Z_{\geq 1}$ is defined by
\begin{align*}
	D_{\mu,i}(f) = \lim_{r \rightarrow 1^-}\frac{1}{\pi i!(i-1)!}\int_{r\D}|f^{(i)}(z)|^2 P_\mu(z) (1-|z|^2)^{i-1} dA(z).
\end{align*}
For $i = 0$ we define
\begin{align*}
	D_{\mu,0}(f) = \lim_{r \rightarrow 1^-}\int_{\T}|f(rz)|^2 P_\mu(rz)\frac{|dz|}{2\pi},
\end{align*}
and for $i < 0$ we use the convention that $D_{\mu,i} (f) = 0$.

Let $\vec{\mu} = (\mu_0, \ldots, \mu_{n-1}) \in (\HD')^n$, and define the quadratic form
\begin{align*}
	\|f\|_{\vec{\mu}}^2 = \sum_{i=0}^{n-1} D_{\mu_i,i}(f), f \in \HD_a.
\end{align*}
Note that $\|1\|_{\vec{\mu}}^2=\hat \mu_0(0)$. If $\|1\|_{\vec{\mu}}^2 = 1$ and there exists $C > 0$ such that
\begin{align*}
	0 \leq \|zf\|_{\vec{\mu}}^2 \leq C \|f\|_{\vec{\mu}}^2, f \in \HD_a,
\end{align*}
then we say that $\vec{\mu}$ is a \emph{normalized allowable $n$-tuple}. For any such tuple, let
\begin{align*}
	\HK_{\vec{\mu}} = \{f\in\HD_a\mid \|f\|_{\vec{\mu}}=0\},
\end{align*}
and define $\HD_{\vec{\mu}}$ as the completion of $\HD_a/\HK_{\vec{\mu}}$ with respect to the norm (induced by) $\|\cdot\|_{\vec{\mu}}$. Since $M_z$ is well defined on $\HD_a/\HK_{\vec{\mu}}$, it extends to be a bounded linear operator on $\HD_{\vec{\mu}}$.

We are now ready to state the model theorem from \cite{Ry19}.
\begin{theorem}[\cite{Ry19}]\label{theorem:RydheModel}
If $\vec{\mu} \in (\HD')^n$ is a normalized allowable $n$-tuple, then $T_{\vec{\mu}}=(M_z, \HD_{\vec{\mu}})$ is a bounded $n$-isometry, and $1$ is a cyclic unit vector for $T_{\vec{\mu}}$.

Conversely, if $T \in \HB(\HH)$ is an $n$-isometry with a cyclic unit vector $e$, then $\vec{\mu} \in (\HD')^n$ given by
\begin{align}\label{eq:DefinitionOfAllowableTuple}
	\hat{\mu}_i(k) = \overline{\hat{\mu}_i(-k)} = \langle \Delta^{(i)}e, T^k e\rangle_\HH, \quad k \in \Z_{\geq 0}
\end{align}
is a normalized allowable $n$-tuple, and there exists a unitary map $U: \HH \rightarrow \HD_{\vec{\mu}}$ such that $UT = T_{\vec{\mu}} U$ and $U e = 1$.

If $T_j: \HH_j \rightarrow \HH_j, j =1, 2$ are bounded $n$-isometries with cyclic vectors $e_j$, then the associated $n$-tuple $\vec{\mu}_j$ coincide if and only if there exists a unitary map $U: \HH_1 \rightarrow \HH_2$ such that $T_1 = U^* T_2 U$ and $Ue_1 = e_2$.
\end{theorem}

When $\vec{\mu} = (\mu_0, \ldots, \mu_{n-1}) \in (\HD')^n$ is a normalized allowable $n$-tuple, it is known that the leading distribution $\mu_{n-1}$ is a nonnegative measure (e.g. \cite[Proposition 2.1]{Ry19}). Moreover, it follows from \eqref{eq:DefinitionOfAllowableTuple} that $\mu_{n-1}(|p|^2)=\La \Delta^{(n-1)} p(T)e,p(T)e\Ra$ for any polynomial $p$. Since $e$ is cyclic, it follows that $\Delta^{(n-1)}\ge 0$, and $T$ is a strict $n$-isometry if and only if $\mu_{n-1}\ne 0$. The remaining distributions $\mu_0,\ldots,\mu_{n-2}$ may be less regular (a simple example is provided by Corollary~\ref{corollary:Dlambdam=Dvecmu}).

We note that the case $\mu_0=\frac{|dz|}{2\pi}$ corresponds precisely to the case where the cyclic unit vector $e\in\ker T^*$. Recalling that any $n$-isometry is bounded below, one can show that if $e\in\ker T^*$, then $\|\cdot\|_{\vec{\mu}}$ defines a proper norm on $\HD_a$, e.g. \cite[Proposition~7.3]{Ry19}. If $e\in\ker T^*$ and $T$ is expansive, then we even obtain that $\HD_{\vec{\mu}}\subseteq H^2$ in the sense of a contractive embedding.

\subsection{$n$-isometric shifts on $\HH(b)$} Recall that if $b\in\HS_0(\C,\C)$, then $M_z\HH(b)\subseteq\HH(b)$ if and only if $b$ satisfies \eqref{eq:MateExistenceCondition}. In this case $(M_z,\HH(b))$ is bounded and expansive. We now outline some results from \cite[Section~9]{LGR} on when $(M_z,\HH(b))$ is a higher order isometry. The notion of \emph{higher order local Dirichlet spaces} will be central to this discussion.

If $\lambda\in \T$, then  the local Dirichlet space of order $m \in \Z_{\geq 1}$ at $\lambda$ is defined by
\begin{align*}
	\HD_\lambda^m=\{p+(z-\lambda)^mh: p \text{ is a polynomial of degree }<m \text{ and } h\in H^2\}.
\end{align*}
The norm on $\HD^m_\lambda$ is defined by
\begin{align*}
	\|p+(z-\lambda)^mh\|_{\HD_\lambda^m}^2=\|p+(z-\lambda)^mh\|_{H^2}^2+\|h\|_{H^2}^2.
\end{align*}
If $f\in \HD_\lambda^m$, then the derivatives $f^{(j)}(\lambda)$ exist in the sense of radial limits for $0\le j\le m-1$. This allows us to consider the Taylor polynomial
\begin{align*}
	T_{m-1}(f,\lambda)(z)=\sum_{j=0}^{m-1} \frac{f^{(j)}(\lambda)}{j!}(z-\lambda)^j.
\end{align*}
It holds that $f\in \HD_\lambda^m$ if and only if $T_{m-1}(f,\lambda)$ exists and the higher order local Dirichlet integral
\begin{align} \label{localDiri}
	D_\lambda^m(f)= \int_{|z|=1} \left| \frac{f(z)-T_{m-1}(f,\lambda)(z)}{(z-\lambda)^m}\right|^2\frac{|dz|}{2\pi}
\end{align}
is finite. In this case, $f(z)=T_{m-1}(f,\lambda)(z)+(z-\lambda)^mh(z)$, and $D_\lambda^m(f)=\|h\|_{H^2}^2$. We note that if $m=1$, then we obtain the local Dirichlet integral $D_\lambda^1(f)=D_\lambda(f)$ discussed above.

If $T$ is an $n$-isometry, but not an $(n-1)$-isometry, then we call $T$ a \emph{strict} $n$-isometry.
\begin{theorem}[{\cite[Theorem~1.1]{LGR}}]\label{theorem:n-isometriesOnH(b)} Suppose $b\in\HS_0(\C,\C)$ satisfies \eqref{eq:MateExistenceCondition}, and let $m\in \Z_{\geq 1}$. Then $(M_z,\HH(b))$ is not a strict $(2m+1)$-isometry. Moreover, the following are equivalent:
\begin{enumerate}
\item $(M_z,\HH(b))$ is a strict $2m$-isometry.
\item There exists $\lambda\in \T$ and a polynomial $p$ of degree $< m$ such that $p(\lambda) \ne 0$ and
\begin{align*}
	\|f\|^2_{\HH(b)}=\|f\|^2_{H^2}+D_\lambda^m(pf).
\end{align*}
\end{enumerate}
If the above conditions hold, then $\HH(b)=\HD_\lambda^m$ with equivalence of norms.
\end{theorem}

\subsection{The relation between $\HH(b)$ and $\HD_{\vec{\mu}}$}

It's an exercise to show that
\begin{align}\label{eq:DeltaRecurssion}
	\Delta^{(i)}=\sum_{l=0}^{i-1}(-1)^{i-1-l}\binom{i-1}{l}T^{*l}\Delta T^l
\end{align}
for $i\in\Z_{\ge 1}$, cf. \cite[Proposition~2.1]{Ry19}. We will use this to prove Theorem~\ref{theorem:H(b)=Dvecmu}, which is the technical main result of this section.

\begin{proof}[Proof of Theorem \ref{theorem:H(b)=Dvecmu}]
	If $b\in\HS_0(\C,\C)$ satisfies \eqref{eq:MateExistenceCondition} and $T=(M_z,\HH(b))$ is a strict $n$-isometry, then $n=2m$ by Theorem~\ref{theorem:n-isometriesOnH(b)}. Moreover, here exists $\lambda\in \T$ and a polynomial $p(z)=\sum_{j=0}^{m-1}c_jz^j$ with $p(\lambda)\ne 0$ such that
	\begin{align*}
		\|f\|^2_{\HH(b)}=\|f\|^2_{H^2}+D_\lambda^m(pf).
	\end{align*}
	By \cite[Theorem~5.5]{AM19}, polynomials are dense in $\HH(b)$, in other words $1$ is cyclic for $T$. Since $b(0)=0$, $\|1\|_{\HH(b)}=1$. By Theorem~\ref{theorem:RydheModel}, $T$ is unitarily equivalent to $T_{\vec{\mu}}=(M_z,\HD_{\vec{\mu}})$, where $\vec{\mu}$ is the normalized allowable $2m$-tuple given by \eqref{eq:DefinitionOfAllowableTuple}, and the unitary equivalence preserves the cyclic vector $1$. It follows from $b(0)=0$ that $1\in\ker T^*$, and therefore $\mu_0=\frac{|dz|}{2\pi}$. Similar to the proof of Theorem~\ref{dirichlettype} we obtain that $\HH(b)=\HD_{\vec{\mu}}$.
	
	To conclude the proof of the direct part of the theorem, it remains to show that $\mu_1,\ldots,\mu_{2m-1}$ are given by \eqref{eq:H(b)=Dvecmu}. It follows from \cite[Lemma~9.3]{LGR} that
	\begin{align*}
		\La \Delta f,g\Ra_{\HH(b)}=\frac{(pf)^{(m-1)}(\lambda)\overline{(pg)^{(m-1)}(\lambda)}}{[(m-1)!]^2}.
	\end{align*}
	In combination with \eqref{eq:DeltaRecurssion}, we can now compute $\mu_i$ in accordance with \eqref{eq:DefinitionOfAllowableTuple}, namely for $k\in\Z_{\geq 0}$ it holds that
	\begin{align*}
		\hat\mu_i(k)&=\La \Delta^{(i)} 1, T^k 1\Ra_{\HH(b)}
		\\
		&=
		\sum_{l=0}^{i-1}(-1)^{i-1-l}\binom{i-1}{l}\La \Delta T^l1, T^lT^k 1\Ra_{\HH(b)}
		\\
		&=
		\sum_{l=0}^{i-1}(-1)^{i-1-l}\binom{i-1}{l}\frac{(z^lp)^{(m-1)}(\lambda)\overline{(z^{k+l}p)^{(m-1)}(\lambda)}}{[(m-1)!]^2}.
	\end{align*}
	It is now an algebraic exercise to show that the above right-hand side agrees with \eqref{eq:H(b)=Dvecmu}.
	
	For the converse direction, suppose $\lambda\in\T$, $p(z)=\sum_{j=0}^{m-1}c_jz^j$ with $p(\lambda)\ne 0$, and $\vec{\mu}=(\frac{|dz|}{2\pi},\mu_1,\ldots,\mu_{2m-1})$ is given by \eqref{eq:H(b)=Dvecmu}. By Proposition \ref{proposition:HequivalenttoHB}, there exists $b\in\HS_0(\C,\C)$ such that \eqref{eq:MateExistenceCondition} is satisfied and
	\begin{align*}
		\|f\|_{\HH(b)}^2=\|f\|_{H^2}^2+D_\lambda^m(pf).
	\end{align*}
	Note that by Theorem \ref{theorem:n-isometriesOnH(b)}, $(M_z,\HH(b))$ is a strict $2m$-isometry. Since the corresponding $\vec{\mu}$ is given by \eqref{eq:H(b)=Dvecmu}, we see that $\vec{\mu}$ is a normalized allowable tuple.
\end{proof}

Recall that differentiation $\mu\mapsto\partial\mu$ in $\HD'$ corresponds to the Fourier multiplier $\hat\mu(k)\mapsto k\hat\mu(k)$. In the context of the above theorem, it is natural to consider the modified differentiation operator $D$ corresponding to $\hat\mu(k)\mapsto|k|\hat\mu(k)$. Indeed, since $\hat\delta_\lambda(k)=\overline{\lambda}^k$, and the right-hand side of \eqref{eq:H(b)=Dvecmu} takes the shape $p(k)\overline{\lambda}^k$ for some polynomial $p$, the distributions associated to a $2m$-isometric shift $(M_z,\HH(b))$ via Theorem~\ref{theorem:RydheModel} are given by $p(D)\delta_\lambda$. In the next result, we make this idea precise in the ``simplest'' case $\HH(b)=\HD_\lambda^m$.

Let $\mu \in \HD'$, then there exists $N \in \Z_{\geq 0}$ such that $|\hat{\mu}(k)| = |\mu(\zeta^{-k})| \lesssim (1+|k|)^N, k \in \Z$. The smallest such $N$ is called the order of $\mu$. In the following, we use the convention that $\binom{n}{j} = 0$ whenever $j < 0$ or $j > n$.
\begin{corollary}\label{corollary:Dlambdam=Dvecmu}
	Let $m\in\Z_{\ge 1}$ and $\lambda\in\T$. The normalized allowable tuple $\vec{\mu} = (\frac{|dz|}{2\pi}, \mu_1 \ldots, \mu_{2m-1})$ corresponding to the expansive strict $2m$-isometry $(M_z,\HD_\lambda^m)$ is given by
	\begin{align*}
		\mu_i =\binom{i-1}{m-1}\frac{\prod_{j=i+1-m}^{m-1}(D+j)}{(2m-1-i)!}\delta_\lambda
	\end{align*}
	for $1\le i \le 2m-2$, and $\mu_{2m-1} =\binom{2m-2}{m-1}\delta_\lambda$. In particular, $\mu_i=0$ for $1\le i\le m-1$, and $\mu_i$ is of order $2m-1-i$ for $m\le i\le 2m-1$.
\end{corollary}

\begin{proof}
	According to Theorem~\ref{theorem:H(b)=Dvecmu}, the distributions $\mu_1,\ldots,\mu_{2m-1}$ are given by \eqref{eq:H(b)=Dvecmu}, with $c_0=1$, $c_1,\ldots,c_{m-1}=0$. In other words,
	\begin{align*}
		\hat\mu_i(k)=\overline{\lambda}^k\sum_{l=0}^{i-1}(-1)^{i-1-l}\binom{i-1}{l}\binom{l}{m-1}\binom{l+k}{m-1},\quad k\ge 0.
	\end{align*}
	Since all terms with $l<m-1$ vanish, it is immediate that $\mu_1,\ldots,\mu_{m-1}=0$. We henceforth consider $m\le i \le 2m-1$. One easily verifies that
	\begin{align*}
		\binom{i-1}{l}\binom{l}{m-1}=\binom{i-1}{m-1}\binom{i-m}{i-1-l}.
	\end{align*}
For $k\ge 0$, the $k$-th Fourier coefficient of $\prod_{j=i+1-m}^{m-1}(D+j)\delta_\lambda$ is $\prod_{j=i+1-m}^{m-1}(k+j) \overline{\lambda}^k$, so in order to finish the proof it suffices to establish the identity
	\begin{align*}
		\sum_{l=m-1}^{i-1}(-1)^{i-1-l}\binom{i-m}{i-1-l}\binom{l+k}{m-1}=\binom{m+k-1}{2m-1-i}.
	\end{align*}
	To this end, denote the above left- and right-hand sides by $\mathrm{LHS}(i,k,m)$ and $\mathrm{RHS}(i,k,m)$. It is clear that $\mathrm{LHS}(m,k,m)=\mathrm{RHS}(m,k,m)$. The elementary identity
	\begin{align*}
		\binom{N}{M-1}+\binom{N-1}{M-1}=\binom{N}{M}
	\end{align*}
	and a rearrangement of terms yields that
	\begin{align*}
		\mathrm{LHS}(i,k+1,m)-\mathrm{LHS}(i,k,m)
		&=
		\sum_{l=m-1}^{i}(-1)^{i-l}\left[\binom{i-m}{i-l}+\binom{i-m}{i-1-l}\right]\binom{l+k}{m-1}
		\\
		&=
		\sum_{l=m-1}^{i}(-1)^{i-l}\binom{i+1-m}{i-l}\binom{l+k}{m-1}
		\\
		&=
		\mathrm{LHS}(i+1,k,m).
	\end{align*}
	By simple calculation,
	\[\mathrm{RHS}(i,k+1,m)-\mathrm{RHS}(i,k,m)=\mathrm{RHS}(i+1,k,m),\]
	and the desired identity follows by induction with respect to $i$.
\end{proof}

\subsection{The relation between $\HH[B]$ and $\HD_{\vec{\mu}}$}

Suppose $B\in\HS_0(\ell^2,\C)$, $\rank B<\infty$, and $\HH[B]$ is $M_z$-invariant. Let $T=(M_z,\HH[B])$ and $\Delta = T^*T-I$. Note that $\rank B = \rank\Delta $. In particular, if $\rank\Delta=1$ and $T$ is an $n$-isometry, then $\HH[B]=\HH(b)$ is a higher order local Dirichlet space equipped with an equivalent norm according to Theorem~\ref{theorem:n-isometriesOnH(b)}. For higher finite ranks of $\Delta$, $\HH[B]$ instead is a finite intersection of such spaces.

\begin{theorem}[{\cite[Theorem~1.5]{LGR}}]\label{MainTheoremfinite} Let $B\in\HS_0(\ell^2,\C)$ be such that $T=(M_z,\HH(B))$ is bounded and satisfies $\rank \Delta <\infty$.

Then $T$ is a $2m$-isometry, if and only if there are $\lambda_1, \dots, \lambda_k \in \T$, pairs of integers $(m_1,n_1), \dots (m_k, n_k)$ such that $1 \le n_j\le m_j \le m$ for all $j$,   and there are polynomials $\{p_{ij}\}_{1\le j\le k, 1\le i\le n_j}$ with degree $p_{ij} \le m_j-1$ for $1\le j \le k$, $1\le i\le n_j$  such that

\begin{align}\label{2mIsoNorm}\|f\|^2_{\HH(B)}=\|f\|^2_{H^2} + \sum_{j=1}^k\sum_{i=1}^{n_j} D^{m_j}_{\lambda_j}(p_{ij}f).\end{align}

There is a choice of parameters so that $p_{1j}(\lambda_j)\ne 0$  for each $j$ and such that $\sum_{j=1}^kn_j=\rank \Delta$. If all that is the case, then $\HH(B)=\bigcap_{j=1}^k \HD_{w_j}^{m_j}$ with equivalence of norms.
\end{theorem}

In the setting of the above theorem, we have
\begin{align*}
	\La\Delta f,g\Ra_{\HH[B]}
	=
	\sum_{j=1}^k\sum_{i=1}^{n_j}\frac{(p_{ij}f)^{(m_j-1)}(\lambda_j)\overline{(p_{ij}g)^{(m-1)}(\lambda_j)}}{[(m_j-1)!]^2}.
\end{align*}
So proceeding just as in the proof of Theorem~\ref{theorem:H(b)=Dvecmu}, we can now use \eqref{eq:H(b)=Dvecmu} to calculate the normalized allowable tuple $\vec\mu=(\frac{|dz|}{2\pi},\mu_1,\ldots,\mu_{2m-1})$ such that $\HD_{\vec{\mu}}=\HH[B]$. Before we state the corresponding higher rank analogue of Theorem~\ref{theorem:H(b)=Dvecmu}, we attempt to introduce a more compact notation. For $N\in\Z_{\ge 1}$, $k\in\Z_{\geq 0}$ and fixed $m\in\Z_{\ge 1}$, consider the $m\times m$-matrices
\begin{align}\label{eq:HM_N(k)}
	\HM_N(k)=\left(\sum_{l=0}^{N-1}(-1)^{N-1-l}\binom{N-1}{l}\binom{l+j_1}{m-1}\binom{l+k+j_2}{m-1}\right)_{0\le j_1,j_2\le m-1}.
\end{align}
To each polynomial $p(z)=\sum_{j=0}^{m-1}c_jz^j$ we associate the vector of monomials $\vec p(z)=(c_jz^j)_{0\le j\le m-1}$. With this notation, \eqref{eq:H(b)=Dvecmu} takes the form
\begin{align*}
	\hat\mu_{N}(k)=\overline{\hat \mu_N(-k)}=\overline{\lambda}^k\La \HM_N(k)\vec p(\lambda),\vec p(\lambda)\Ra_{\C^m}.
\end{align*}
Using the same argument as in the proof of Theorem \ref{theorem:H(b)=Dvecmu}, we obtain the following result.
\begin{theorem}\label{theorem:H[B]=Dvecmu}
	Let $\HM_{N}(k)$ be the matrices given by \eqref{eq:HM_N(k)}.
	
	Suppose $B\in\HS_0(\ell^2,\C)$, $\rank B<\infty$, and $\HH[B]$ is $M_z$-invariant. If $T=(M_z,\HH[B])$ is a $2m$-isometry for some $m\in\Z_{\ge 1}$, then there exists $\lambda_1,\ldots,\lambda_k\in\T$, positive integers $m_1,\ldots,m_k$ and $n_1,\ldots,n_k$ with $1\le n_j\le m_j\le m$, and polynomials $\{p_{ij}\}_{1\le j \le k,1\le i\le n_j}$ with the following properties: $\deg p_{ij}\le m_j-1$ and $p_{1j}(\lambda)\ne 0$ for $1\le j\le k$, and if $\vec{\mu}=(\frac{|dz|}{2\pi},\mu_1,\ldots,\mu_{2m-1})$ is given by
	\begin{align}\label{eq:H[B]=Dvecmu}
		\hat \mu_{N}(l)=\overline{\hat \mu_{N}(-l)}
		=
		\sum_{j=1}^k\overline{\lambda}_j^l\sum_{i=1}^{n_j}\La \HM_N(l)\vec{p}_{ij}(\lambda_j),\vec{p}_{ij}(\lambda_j)\Ra_{\C^m},
	\end{align}
	where $l\in\Z_{\geq 0}$ and $1\le N\le 2m-1$, then $\HH[B]=\HD_{\vec{\mu}}$.
	
	Conversely, if $m\in\Z_{\ge 1}$, $\lambda_1,\ldots,\lambda_k\in\T$, $m_1,\ldots,m_k$ and $n_1,\ldots,n_k$ are positive integers with $1\le n_j\le m_j\le m$, $\{p_{ij}\}_{1\le j \le k,1\le i\le n_j}$ are polynomials with $\deg p_{ij}\le m_j-1$ and $p_{1j}(\lambda)\ne 0$ for $1\le j\le k$, and if $\vec{\mu}=(\frac{|dz|}{2\pi},\mu_1,\ldots,\mu_{2m-1})$ is given by \eqref{eq:H[B]=Dvecmu}, then $\vec{\mu}$ is a normalized allowable $2m$-tuple, and there exists a rational $B\in\HS_0(\ell^2,\C)$ with $\rank B<\infty$ such that $\HH[B]=\HD_{\vec{\mu}}$.
\end{theorem}

We now present a condition for the defect operator on a $\HD_{\vec{\mu}}$ space to be of finite rank.
\begin{proposition}\label{basictwomiso}
Let $\vec{\mu} = (\frac{|dz|}{2\pi}, \mu_1 \ldots, \mu_{n-1}) \in (\HD')^{n}$ $(n \geq 2)$ be a normalized allowable $n$-tuple, and $\mu_{n-1} \neq 0$. Then $T = (M_z, \HD_{\vec{\mu}})$ is expansive with $\rank \Delta < \infty$, if and only if $T$ is expansive, $n = 2m$ for some $m \geq 1$, and there is a polynomial $p(z) = \prod_{j=1}^k (z-\lambda_j)^{m_j}$ with $k \leq \rank \Delta$, $\lambda_j \in\T, j = 1, \ldots, k$, $m = \max\{m_1, \ldots, m_k\}$ such that the following holds
\begin{align}\label{finiterank}
\sum_{s=1}^{2m-1} D_{\mu_s,s-1}(p(T)q) = 0, ~~q \text{ polynomials},
\end{align}
and (\ref{finiterank}) doesn't hold if $p$ is replaced by any factor $p_1$ of $p$ with $\deg p_1 < \deg p$.
\end{proposition}
\begin{proof}
By \cite[Proposition 3.4]{Ry19},
\begin{align}\label{sumdelta}
\sum_{s=1}^{2m-1} D_{\mu_s,s-1}(p(T)q) = \langle \Delta (pq), pq\rangle, ~~q \text{ polynomials}.
\end{align}
If $T $ is expansive with $\rank \Delta < \infty$, then by \cite[Theorem 1.4]{LGR} $n$ is even, say $n = 2m$, and $\HN: = [\ran \Delta]_{T^*}$ is finite dimensional. Suppose $\tilde{p}(z) = \prod_{j=1}^k (z-\overline{\lambda_j})^{m_j}$ is the minimal polynomial of $T^*|\HN$, then $\tilde{p}(T^*)\Delta = 0$. By \cite[Theorem 8.10]{LGR} $k \leq \rank \Delta$, $\lambda_j \in\T, j = 1, \ldots, k$, $m = \max\{m_1, \ldots, m_k\}$. Let $p(z) = \overline{\tilde{p}(\overline{z})}$. Note that $\tilde{p}(T^*)\Delta = 0$, if and only if $\Delta p(T) = 0$, which is equivalent to $\Delta^{1/2} p(T) = 0$. Thus (\ref{finiterank}) follows from (\ref{sumdelta}).

Conversely, if $T = (M_z, \HD_{\vec{\mu}})$ is expansive, $n = 2m$ and (\ref{finiterank}) holds, then $\Delta \geq 0$ and $\|\Delta^{1/2} (pq)\| = 0$, $q$ polynomials. Since by Theorem \ref{theorem:RydheModel} polynomials are dense in $\HD_{\vec{\mu}}$, the closure of $\{pq: q \text{ polynomials}\}$ has codimension $\leq \deg p$ in $\HD_{\vec{\mu}}$, we conclude that $\Delta$ has finite rank. In this case, by assumption $p$ is the minimal polynomial of $P_\HN T|\HN$.
\end{proof}

\begin{corollary}\label{infiniterank}
Let $\vec{\mu} = (\frac{|dz|}{2\pi}, \mu_1 \ldots, \mu_{n-1}) \in (\HD')^n$ $(n \geq 2)$ with $\mu_i$ positive measures, $1 \leq i \leq n-1$ and $\mu_{n-1} \neq 0$. Then
\begin{itemize}
\item[(i)] $T = (M_z, \HD_{\vec{\mu}})$ is expansive.
\item[(ii)] $\Delta = T^*T - I$ has finite rank if and only if $n = 2$ and $\mu_1$ is a finitely atomic measure.
\end{itemize}
\end{corollary}
\begin{proof}
(i) This is clear.

(ii) If $n = 2$ and $\mu_1$ is a finitely atomic measure, then it follows from (\ref{eq:DirichletMzIdentity}) that $\Delta = T^*T - I$ has finite rank.

If $n \geq 3$, then (\ref{finiterank}) doesn't hold for any polynomial $p$. Thus $\Delta = T^*T - I$ has infinite rank if $n \geq 3$.  So if $\Delta = T^*T - I$ has finite rank, then $n =2$. Then Theorem \ref{theorem:HBDmuequal} implies that $\mu_1$ is a finitely atomic measure.
\end{proof}

\begin{remark}
	The case considered above, where each $\mu_i$ is a positive measure, has also been studied in~\cite{GGR}, where a complete characterization in terms of operator inequalities is given.
\end{remark}

\subsection{Examples}

It may happen that the support of $\mu_{2m-1}$ in Theorem \ref{theorem:H[B]=Dvecmu} is strictly contained in $\{\lambda_1, \ldots, \lambda_k\}$.
\begin{example}\label{supportstrict}
Let $\vec{\mu} = (\frac{|dz|}{2\pi}, \delta_1, (D+1)\delta_{-1}, 2\delta_{-1})$. Then $\vec{\mu} = \vec{\mu}^1 + \vec{\mu}^2-(\frac{|dz|}{2\pi}, 0, 0, 0)$, where $\vec{\mu}^1 = (\frac{|dz|}{2\pi}, \delta_1, 0, 0), \vec{\mu}^2 = (\frac{|dz|}{2\pi}, 0, (D+1)\delta_{-1}, 2\delta_{-1})$. Thus by (\ref{eq:DirichletMzIdentity}) and Corollary \ref{corollary:Dlambdam=Dvecmu},
$$\langle \Delta f, f\rangle = |f(1)|^2 + |f'(-1)|^2, f \text{ polynomials}.$$
So $\Delta = K_1^{\vec{\mu}} \otimes K_1^{\vec{\mu}} + \overline{\partial} K_{-1}^{\vec{\mu}}\otimes \overline{\partial} K_{-1}^{\vec{\mu}}\geq 0$, and $\Delta$ has rank $2$, where $\overline{\partial} K_{w}^{\vec{\mu}} = \frac{\overline{\partial} K_{w}^{\vec{\mu}}}{\partial\overline{w}}$. But $\mu_3 = 2 \delta_{-1}$ has support $\{-1\} \subsetneq \{-1, 1\}$.
\end{example}

It is now easy to give a tuple of distributions $\vec{\mu}$ such that $(M_z, \HD_{\vec{\mu}})$ is expansive and the defect operator is of infinite rank.
\begin{example}\label{conversenothold}
$(i)$ Let $\vec{\mu} = (\frac{|dz|}{2\pi}, 0, 0, \delta_1)$. Then by Corollary \ref{infiniterank} $\rank \Delta = \infty$.

$(ii)$ Let $\vec{\mu} = (\frac{|dz|}{2\pi}, 0, (D+1)\delta_{1}, 3\delta_{1}) = \vec{\mu}^1 + \vec{\mu}^2 - (\frac{|dz|}{2\pi}, 0, 0, 0)$, where $\vec{\mu}^1 = (\frac{|dz|}{2\pi}, 0, (D+1)\delta_{1}, 2\delta_{1})$, $\vec{\mu}^2 = (\frac{|dz|}{2\pi}, 0, 0, \delta_1)$. Note that $(M_z, \HD_{\vec{\mu}})$ is expansive. So $\HD_{\vec{\mu}}$ is equal to some $\HH[B]$. But by part $(i)$ and the discussion in Example \ref{supportstrict} we have $\rank \Delta = \infty$.
\end{example}

The following lemma is well-known, for instance by a combination of \cite[Chapter~VI.3]{Gar} and \cite[Proposition~1.11]{HKZ}.

\begin{lemma}\label{lemma:LittlewoodPaley}
	Let $n\in\Z_{\ge 1}$. An analytic function $f\colon \D\to\C$ belongs to $H^2$ if and only if $f^{(n)}\in L^2(\D,(1-|z|^2)^{2n-1}\,dA(z))$. Moreover,
	\begin{align*}
		\int_{\D}|f^{(n)}(z)|^2(1-|z|^2)^{2n-1}\, dA(z)\lesssim \|f\|_{H^2}^2.
	\end{align*}
\end{lemma}

\begin{example}
	By Lemma~\ref{lemma:LittlewoodPaley}, the Dirichlet-type integral
	\begin{align*}
		D_{\frac{|dz|}{2\pi},2}(f) =\frac{1}{2\pi}\int_\D |f''(z)|^2 (1-|z|^2)\, dA(z)
	\end{align*}
	is finite if and only if $f'\in H^2$. We note that any such $f$ is continuous on $\T$, see e.g. \cite{GL18}. We use this information to provide an example of a normalized allowable $4$-tuple $\vec{\mu}$ for which $f\in\HD_{\vec{\mu}}$ does not imply that $f'\in H^2$. This gives a negative answer to \cite[Question~7.2]{Ry19}.
	
	The tuple we have in mind is $\vec{\mu} = (\frac{|dz|}{2\pi}, 0, (D+1)\delta_1, 2\delta_1)$. By Corollary~\ref{corollary:Dlambdam=Dvecmu}, $\HD_{\vec{\mu}}=\HD_1^2$. It is immediate from the definition that $\HD_1^2$ contains function which are not continuous on $\T$. In particular, not every $f\in\HD_{\vec{\mu}}$ will satisfy $f'\in H^2$.
\end{example}

\begin{example}
	It was noted in \cite[Example~6.12]{Ry19} that for some allowable $n$-tuples $\vec{\mu}$ with $\mu_{n-1}\ne 0$, the Dirichlet-type integrals $D_{\mu_0,0}(f)$ and $D_{\mu_{n-1},n-1}(f)$ do not control the full norm $\|f\|_{\vec\mu}^2$. The formal statement of this is that
	\begin{align*}
		\sup_{f\in\HD_{\vec{\mu}}\setminus\{0\}}\frac{\|f\|_{\vec\mu}^2}{D_{\mu_0,0}(f)+D_{\mu_{n-1},n-1}(f)}=\infty.
	\end{align*}
	We now complement this with an example of $\vec{\mu}$ for which $\|f\|_{\vec\mu}^2$ does not control $D_{\mu_0,0}(f)$ and $D_{\mu_{n-1},n-1}(f)$, i.e.
	\begin{align}\label{eq:CounterExample}
		\sup_{f\in\HD_{\vec{\mu}}\setminus\{0\}}\frac{D_{\mu_0,0}(f)+D_{\mu_{n-1},n-1}(f)}{\|f\|_{\vec\mu}^2}=\infty.
	\end{align}
	
	Once again we take $\vec{\mu} = (\frac{|dz|}{2\pi}, 0, (D+1)\delta_1, 2\delta_1)$. Then $\HD_{\vec{\mu}}=\HD_1^2=\HH(b)$ for some $b\in\HS_0(\C,\C)$. In particular, $\HD_{\vec{\mu}}$ is contractively contained in $H^2$. Consequently, we may disregard the term $D_{\mu_0,0}(f)=\|f\|_{H^2}^2$ in \eqref{eq:CounterExample}, so if \eqref{eq:CounterExample} does not hold, then $D_{\mu_3,3}(f)\lesssim \|f\|_{\vec{\mu}}^2$. We now show that this leads to a contradiction.
	
	Consider a function $f=(z-1)^2h$, where $h\in H^2$, and recall that
	\begin{align*}
		\|f\|_{\HD_1^2}^2=\|(z-1)^2h\|_{H^2}^2+\|h\|_{H^2}^2\approx \|h\|_{H^2}^2.
	\end{align*}
	The assumption that \eqref{eq:CounterExample} doesn't hold therefore implies that
	\begin{align*}
		D_{\mu_3,3}((z-1)^2h)\lesssim \|h\|_{H^2}^2.
	\end{align*}
	This is in turn equivalent to
	\begin{align*}
		\|h\|_{H^2}^2+D_{\mu_3,3}((z-1)^2h)\lesssim \|h\|_{H^2}^2.
	\end{align*}
	We now show that
	\begin{align}\label{eq:CounterExampleMainClaim}
		\int_\D |(z-1)h^{(3)}(z)|^2(1-|z|^2)^3\, dA(z) \lesssim \|h\|_{H^2}^2+D_{\mu_3,3}((z-1)^2h).
	\end{align}
	Together with the above, this finally leads to the estimate
	\begin{align*}
		\int_\D |(z-1)h^{(3)}(z)|^2(1-|z|^2)^3\, dA(z)\lesssim \|h\|_{H^2}^2.
	\end{align*}
	However, letting $h(z)=\log(1+z)$, the above left-hand side diverges, even though $h\in H^2$. Hence, the assumption that \eqref{eq:CounterExample} is false leads to a contradiction.
	
	It remains to prove the estimate \eqref{eq:CounterExampleMainClaim}. Note that for $f(z)=(z-1)^2h(z)$:
	\begin{align*}
		D_{\mu_3,3}(f)&		=		\frac{1}{12\pi}\int_\D |f^{(3)}(z)|^2\frac{(1-|z|^2)^3}{|z-1|^2}\, dA(z)
		\\
		&=		\frac{1}{12\pi}\int_\D |6h'(z)+6(z-1)h^{(2)}(z)+(z-1)^2h^{(3)}(z)|^2\frac{(1-|z|^2)^3}{|z-1|^2}\, dA(z).
	\end{align*}
Thus
\begin{align*}
& \int_\D |(z-1)h^{(3)}(z)|^2(1-|z|^2)^3\, dA(z)\\
&\lesssim D_{\mu_3,3}(f) + \int_\D |(z-1)h^{(2)}(z)|^2\frac{(1-|z|^2)^3}{|z-1|^2}\, dA(z) + \int_\D |(z-1)h^{(2)}(z)|^2\frac{(1-|z|^2)^3}{|z-1|^2}\, dA(z).
\end{align*}
	Since $\frac{(1-|z|^2)^2}{|z-1|^2}\le 4$ for $z\in\D$, Lemma~\ref{lemma:LittlewoodPaley} implies
	\begin{align*}
		\int_\D |h'(z)|^2\frac{(1-|z|^2)^3}{|z-1|^2}\, dA(z)
		\lesssim
		\int_\D |h'(z)|^2 (1-|z|^2)\, dA(z) \lesssim \|h\|_{H^2}^2.
	\end{align*}
	Similarly,
	\begin{align*}
		\int_\D |(z-1)h^{(2)}(z)|^2\frac{(1-|z|^2)^3}{|z-1|^2}\, dA(z)
		=
		\int_\D |h^{(2)}(z)|^2(1-|z|^2)^3\, dA(z)
		\lesssim
		\|h\|_{H^2}^2.
	\end{align*}
	It thus follows that \eqref{eq:CounterExampleMainClaim} holds.
\end{example}

\begin{example}\label{conditionalconvergent}
	It is shown in \cite[Example~6.13]{Ry19} that there is an allowable tuple $\vec{\mu}$ such that for some $i$ and $f$, the Dirichlet integral $D_{\mu_i,i}(f)$ is conditionally convergent. We now present a normalized allowable $4$-tuple $\vec{\mu}$ such that $(M_z, \HD_{\vec{\mu}})$ is expansive, and $D_{\mu_1,1}(z)$ is conditionally convergent.
	
	By Proposition \ref{proposition:HequivalenttoHB}, there exists $b\in\HS_0(\C,\C)$ such that
	\begin{align*}
		\|f\|^2_{\HH(b)}=\|f\|^2_{H^2}+D_1^2(zf).
	\end{align*}
	By Theorem~\ref{theorem:n-isometriesOnH(b)}, $(M_z,\HH(b))$ is a strict $4$-isometry. It follows from Theorem~\ref{theorem:H(b)=Dvecmu} that $\HH(b)=\HD_{\vec{\mu}}$, where $\vec{\mu}=(\frac{|dz|}{2\pi},\mu_1,\mu_2,\mu_3)$ is given by \eqref{eq:H(b)=Dvecmu}. By calculation, $\mu_1=(D+1)\delta_{1}$, $\mu_2=(D+3)\delta_{1}$, and $\mu_3=2\delta_{1}$.
	
	We now show that
	\begin{align*}
		D_{\mu_1,1}(z) = \lim_{r \rightarrow 1^-}\frac{1}{\pi}\int_{r\D} P_{\mu_1}(z) dA(z)
	\end{align*}
	is conditionally convergent. By calculation,
	\begin{align*}
		P_{\mu_1}(z) = \sum_{k=0}^\infty \hat{\mu}_1(k) z^k + \sum_{k=1}^\infty \hat{\mu}_1(-k) \overline{z}^k = \frac{1}{(1-z)^2} + \frac{1}{(1-\overline{z})^2} - 1.
	\end{align*}
	If $X$ denotes the region $\{z = 1 + r e^{it}\mid 0 < r < \frac{1}{2}, \frac{7\pi}{8} < t < \frac{9\pi}{8}\}$, then
	\begin{align*}
		\int_{\D}|P_{\mu_1}(z)|\, dA(z) & \ge \int_X \left|\frac{1}{(1-z)^2} + \frac{1}{(1-\overline{z})^2}\right|\, dA(z) -\pi\\
		& \gtrsim \int_0^{1/2} r^{-2}\cdot r\, dr - \pi = +\infty.
	\end{align*}
\end{example}

\end{document}